\documentclass[11pt]{amsart}
\usepackage{amsmath,amssymb,amsthm}
\usepackage{hyperref}
\usepackage{mathrsfs}
\usepackage{graphicx}
\usepackage{tikz}
\usepackage{enumerate}
\usepackage{float}

\theoremstyle{plain}
\newtheorem{theorem}{Theorem}[section]
\newtheorem{lemma}[theorem]{Lemma}

\newtheorem{corollary}[theorem]{Corollary}

\theoremstyle{definition}
\newtheorem{definition}[theorem]{Definition}
\newtheorem{remark}[theorem]{Remark}
\newtheorem{example}[theorem]{Example}
\newtheorem{problem}[theorem]{Problem}

\newtheorem*{namedthm*}{\thistheoremname}
\newcommand{\thistheoremname}{} 

\newcommand{\WD}{\mathrm{WD}}
\newcommand{\KF}{\mathrm{KF}}
\newcommand{\cl}{\mathrm{cl}}
\newcommand{\sat}{\mathrm{sat}}

\newcommand{\ua}{\mathord{\uparrow}}
\newcommand{\da}{\mathord{\downarrow}}
\newcommand{\rom}[1]{\rm{\uppercase\expandafter{\romannumeral #1}}}

\makeatletter
\def\ps@pprintTitle{%
  \let\@oddhead\@empty
  \let\@evenhead\@empty
  \def\@oddfoot{\reset@font\hfil\thepage\hfil}
  \let\@evenfoot\@oddfoot
}
\makeatother

\begin{document}

\title{$\omega$-well-filtered spaces, revisited} \thanks{This research is supported by NSFC (No.\,12231007 and No.\,12371457).}

\author{Hualin Miao, Xiaodong Jia, Ao Shen, Qingguo Li}
\address{H.\,Miao, School of Mathematics, Hunan University, Changsha, Hunan, 410082, China. Email: {\rm miaohualinmiao@163.com} }
\address{X.\,Jia (corresponding author),  School of Mathematics, Hunan University, Changsha, Hunan, 410082, China. Email:  {\rm jiaxiaodong@hnu.edu.cn} }
\address{A.\,Shen,  School of Mathematics, Hunan University, Changsha, Hunan, 410082, China. Email:  {\rm shenao2020@163.com} }
\address{Q.\,Li, School of Mathematics, Hunan University, Changsha, Hunan, 410082, China.  Email: {\rm liqingguoli@aliyun.com}}

\begin{abstract}
We prove that a $T_0$ topological space is $\omega$-well-filtered if and only if it does not admit either the natural numbers with the cofinite topology or with the Scott topology as its closed subsets in the strong topology. Based on this,  we offer a refined topological characterization for the $\omega$-well-filterification of $T_0$-spaces and solve a problem posed by Xiaoquan Xu. In the setting of second countable spaces, we also characterise sobriety by convergences of certain $\Pi^0_2$-Cauchy subsets of the spaces. 
\end{abstract}

\maketitle

\section{Introduction}
In Mathematics, various types of ``completions" of different structures have been drawing extensive attention and  playing essential roles both in theory and practice.  Examples include the  Dedekind-MacNeille completion of partially ordered sets, Cauchy sequences completion of metrics and Stone-\v Cech compactification  of topological spaces, et cetera. In domain theory and non-Hausdorff topology,  important completions such as the  $\mathrm{D}$-completion \cite{Keimel}, well-filterification \cite{Wu} and sobrification \cite{GHKLMS} of $T_0$ spaces are particularly well-studied in the form of reflectivity of the corresponding categories, to name a few. 

The notion of  $\omega$-well-filtered spaces, which is strictly weaker than that of well-filtered spaces introduced by Heckmann \cite{Heckmann}, is initially put forward by Xu et\,al.~\cite{XiaoquanXu}. A $T_0$ topological space $X$ is called \emph{$\omega$-well-filtered}, if for any given open subset $U$ and any descending sequence of compact saturated subsets $K_i, i\in \mathbb N$ with $\bigcap_{i=1}^\infty K_i \subseteq U$, one already has $K_i \subseteq U$ for some $i\in \mathbb N$. 
Examples of $\omega$-well-filtered spaces include all well-filtered spaces, hence, all sober spaces. Like well-filtered spaces and sober spaces, $\omega$-well-filtered spaces have many nice properties. For example, the classical result that local compactness and core-compactness coincide on sober spaces can be extended to well-filtered spaces and further to $\omega$-well-filtered spaces. Xu et\,al.~showed in~\cite{XiaoquanXu}  that the category of all $\omega$-well-filtered spaces is a reflective full subcategory of the category of $T_0$ spaces with continuous maps. It reveals the existence of ``$\omega$-well-filterification" for $T_0$ spaces. In the same paper, they gave a characterization of this completion by identifying the corresponding completion space as the family of all $\WD_{\omega}$-subsets endowed with the lower Vietoris topology. 

In this paper, we look more closely at $\omega$-well-filtered spaces through the lens of \emph{Descriptive Set Theory} (DST).  Quite surprisingly, DST has been gradually proven to be a powerful tool in non-Hausdorff topology.  Recently, de Brecht showed that a countably based $T_0$ space is sober if and only if it does not contain a $\Pi_{2}^{0}$-subspace homeomorphic to one of two specific topological spaces $S_1$ or $S_D$~\cite{Brecht}, where $S_1$ and $S_D$ are the natural numbers with the co-finite topology and the Scott topology (in the usual order), respectively.  This result was generalized to first-countable $T_0$ spaces in~\cite{LJMC}, and the authors showed that a first-countable $T_0$ space is sober if and only if it does not contain a $\Pi_{2}^{0}$-subspace homeomorphic to $S_1$, $S_D$ or a directed subset without a  maximum element.  In a similar but different vein, we prove it as a central result of this paper: a $T_0$ space $X$ is $\omega$-well-filtered if and only if $X$ does not contain $S_1$ or $S_D$ as its closed subsets in the strong topology. In the presence of first countability, $X$ is $\omega$-well-filtered if and only if $X$ does not contain copies of $S_1$ or $S_D$ as its $\Pi_{2}^{0}$-subspaces. 

This result provides  $\omega$-well-filtered spaces with a  characterization via forbidden subspaces. As we will show, such a characterization has an immediate application: 
Xu et\,al.~posed in~\cite{XiaoquanXu2} the following problem when they were investigating the relationship between well-filterification and sobrification constructions on topological spaces: 

\begin{problem}\label{rudin}
Is every first-countable $T_0$ space a Rudin space in the sense of Xu?
\end{problem}

\noindent
Based on our aforementioned characterization for $\omega$-well-filtered spaces, we solve Problem~\ref{rudin} in the negative by displaying a counterexample. 

Moreover, in the setting of first-countable spaces, our characterization for $\omega$-well-filtered spaces via forbidden subspaces enables us to give more refined characterizations for $\omega$-well-filterification. Known  $\omega$-well-filterification of $X$ consists of all $\WD_{\omega}$-subsets of $X$ endowed with the lower Vietoris topology. While $\WD_{\omega}$-subsets are defined via quantifying over all continuous maps between $X$ and $\omega$-well-filtered spaces,  our refined description only use information of $X$ and says that $\WD_{\omega}$-subsets of $X$ are precisely the topological closures of copies of $S_1$ or $S_D$ inside $X$ and the closures of singletons of $X$. 
A similar refinement can also be said about the sobrification and  Keimel and Lawson's $\mathrm{D}$-completion. Finally, we provide a novel topological characterization for the standard $\omega$-well-filterification. That is achieved via the aid of a weaker version of the strong topology which we will introduce in this paper.

\section{Preliminaries}
\subsection{Basic notions and notations}

We use standard definitions and notations for domain theory~\cite{GHKLMS} and (non-Hausdorff) topology~\cite{GL}, of which the following may be new to some. 

Given a $T_{0}$ space $(X, \tau)$, we define $x\leq y$ if and only if $x$ is in the topological closure of $\{y\}$, in symbols $x\in \cl(\{y\})$ (also abbreviated as $x\in \cl(y)$). We call this order the \emph{specialization order} on $X$. It is known that $X$ is $T_0$ if only if $X$ is a poset with its specialization order.  A set $K$ of a topological space $X$ is called \emph{saturated} if it is the intersection of its open neighborhoods. Equivalently, $K$ is saturated if and only if  
$$K=\ua K = \{x\in X\mid x \text{~is greater than some~} k\in K \}.$$ 
Let $Y$ be a subspace of $X$ and $K\subseteq Y$. We write $\ua K\cap Y=\ua_Y K$, and dually $\da K\cap Y=\da_K Y$.  For $A \subseteq X$, the intersection of all open neighborhoods of $A$ is a saturated set. It is called the \emph{saturation} of $A$ and denoted by $\sat A$. The set of all compact saturated subsets of $X$ is denoted by $Q(X)$. 

Let $X$ be a $T_0$ space. We denote the set of all closed sets of $X$ by $\Gamma (X)$ and all open sets by $\mathcal{O} (X)$. The \emph{strong topology} (also called the Skula topology \cite{Skula}) of $X$ has as a subbasis of open sets $\Gamma(X)\cup \mathcal{O}(X)$. A subset of $X$ is called \emph{Skula closed} in $X$ if it is closed in the strong topology of $X$. One sees immediately that a subset $A$ of $X$ is Skula closed if and only if $A$ is of the form $\bigcap_{i\in I}(A_{i}\cup U_{i})$ for some indexed set $I$, where $A_i\in \Gamma(X)$ and $U_i\in \mathcal{O}(X)$.  
For any set $ A\subseteq X$, we denote the closure of $A$ by $\cl(A)$, and for any set $A \subseteq Y\subseteq X$, we denote the closure of $A$ inside $Y$ (treated as a subspace of $X$) by $\cl_Y(A)$. 
A topological space $X$ is second-countable $(A_2)$ if and only if its topology has a countable basis, and it is said to be first-countable $(A_1)$ if each point of $X$ has a countable neighbourhood basis (local base).


A \emph{Scott open} subset of a poset $P$ is an upper set $U$ of $P$ such that, for every directed subset $D$ of $P$ such that $\sup D$ exists and is in $U$, there is a $d\in D$ such that $d \in U$. The collection of all Scott open subsets of $P$ will be called the \emph{Scott topology} of $P$.

A $T_0$ space $X$ is called a \emph{monotone convergence space} (or a $d$-space), if $X$ is a dcpo in the specialization order and every open subset of $X$ is Scott open in, again, the specialization order. On a $d$-space, the \emph{$d$-topology} consists of, as its closed subsets, sub-dcpo's of $X$. 

\subsection{Quasi-Polish spaces}\label{qps-section}

Quasi-Polish spaces were initially considered by de Brecht as non-Hausdorff extensions for Polish spaces \cite{B}, and form an interesting class of spaces. While the topologies on Polish spaces are induced by complete metrics, the topologies on quasi-Polish spaces are induced by \emph{quasi-metrics} that are also complete in a certain sense.  
A \emph{quasi-metric} on a set $X$ is a function $d : X \times  X \rightarrow  [0,\infty )$ such that for all $x, y, z \in  X$:

\begin{enumerate}
\item $x = y\Leftrightarrow d(x, y) = d(y, x) = 0$, 
\item $d(x, z)\leq d(x, y) +d(y, z)$. 
\end{enumerate}
A quasi-metric space is a pair $(X,d)$, where $d$ is a quasi-metric on $X$. 

A quasi-metric $d$ on $X$ induces a $T_{0}$ topology $\tau _d$ on $X$ generated by basic open balls of the form $B_{d}(x, \varepsilon) =\{y \in X \mid d(x, y) < \varepsilon \}$ for $x \in X$ and a real number $\varepsilon  > 0$.
If $(X,d)$ is a quasi-metric space, then $(X,\hat{d})$ is a metric space, where $\hat{d}$ is defined as $\hat d(x,y) = \max\{d(x,y),d(y,x)\}$. The metric topology induced by $\hat{d}$ will be denoted by $\tau_{\hat{d}}$.

A sequence $(x_n)_{n\in \omega}$ in a quasi-metric space $(X,d)$ is \emph{Cauchy} if and only if for each real number $ \varepsilon  > 0$, there exists $n_0 \in \omega $ such that $d(x_n, x_m) < \varepsilon $ for all $m\geq n\geq n_0$. $(X,d)$ is a \emph{complete} quasi-metric space if and only if every Cauchy sequence in $X$ converges with respect to the metric topology $\tau_{\hat{d}}$.

\begin{definition}\cite{B}
A topological space is \emph{quasi-Polish} if and only if it is countably based and completely quasi-metrizable. That is, the topology is induced by a complete quasi-metric defined on the space.
\end{definition}

We will need the following special sets for our reasoning. 
\begin{definition}
Let $X$ be a topological space. For each ordinal $\alpha$ $(1\leq \alpha < \omega_{1} )$, we define $\Sigma^{0}_{\alpha} (X)$ inductively as follows.
\begin{enumerate}
\item $\Sigma^{0}_{1}(X)$ is the set of all open subsets of $X$.
\item For $\alpha > 1, \Sigma^{0}_{\alpha} (X)$ is the set of all subsets $A$ of $X$ which can be expressed in the form
	$A=\bigcup_{i\in \mathbb{N}}B_{i}\backslash B_{i}^{'}$, where for each $i$, $B_{i}$ and $B_{i}^{'}$ are in $\Sigma_{\beta_{i}}^{0}$ for some $\beta_{i}<\alpha$.
\item Define $\Pi^{0}_{\alpha}(X)=\{X \backslash A\mid A \in \Sigma^{0}_{\alpha}(X)\}$ and $\Delta^{0}_{\alpha}(X)=\Sigma^{0 }_{\alpha} (X)\cap \Pi^{ 0}_{\alpha} (X)$.
\end{enumerate}
\end{definition}

From the above definition, we know that $\Pi^{0}_{2}(X)=\{\bigcap_{n\in \mathbb{N}}(A_{n}\cup U_{n})\mid A_n\in \Gamma(X), U_{n}\in \mathcal{O}(X)\}$. Obviously, $A$ is closed in the strong topology of $X$ for every $A\in \Pi^{0}_{2}(X)$.

\subsection{Sober spaces}

A nonempty subset $A$ of a $T_{0}$ space $X$ is \emph{irreducible} if $A \subseteq  B\cup C$ for closed subsets $B$ and $C$ implies $A\subseteq B$ or $A \subseteq C$. A topological space $X$ is \emph{sober} if it is $T_{0}$ and every irreducible closed subset of $X$ is the closure of a (unique) point.
A sobrification of a $T_{0}$ space $X$ consists of a sober space $Y$ and a continuous map $\eta : X \rightarrow Y$ which enjoys the following universal property: For every continuous map $f$ from $X$ to a sober space $Z$, there is a unique continuous map $\overline{f}: Y \rightarrow Z$ such that $f=\overline{f} \circ \eta$.
A standard construction for the sobrification of a $T_{0}$ space $X$ is to set
	$$X^{s} := \{ A\subseteq X : A\in \mathbf{IRR}(X)\}$$
topologized by open sets $U^{ s}:=\{ A \in X ^{s} : A \cap U \neq \emptyset\}$ for each open subset $U$ of $X$. If we define $\eta^{s}_{X} :X \rightarrow X ^{s}$ by $\eta^{s}_{X} ( x ) = \cl(\{ x \}) $, then we obtain a sobrification of $X$ \cite[Exercise~V-4.9]{GHKLMS}, which we call the standard sobrification of $X$.

\subsection{$\omega$-well-filtered spaces}

A $T_{0}$ space $X$ is called \emph{well-filtered}, if for any filtered family $\{K_{i} : i\in I \}\subseteq Q(X)$ and $U \in \mathcal{O}(X)$, the following condition holds,
	$$\bigcap_{i\in I}K_{i}\subseteq U\Rightarrow \exists i_{0}\in I, K_{i_{0}}\subseteq U.$$ 
$X$ is called \emph{$\omega$-well-filtered}, if the above holds for every countable filtered family $\{K_{i} : i < \omega \}\subseteq Q(X)$ and $U \in \mathcal{O}(X)$. Namely,
	$$\bigcap_{i<\omega}K_{i}\subseteq U\Rightarrow \exists i_{0}< \omega, K_{i_{0}}\subseteq U.$$ 
An well-filterification (resp., $\omega$-well-filterification) of a $T_{0}$ space $X$ consists of a well-filtered ($\omega$-well-filtered) space $Y$ and a continuous map $\eta : X \rightarrow Y$ which enjoys the following universal property: for every continuous map $f$ from $X$ to an arbitrary well-filtered (resp., $\omega$-well-filtered) space $Z$, there is a unique continuous map $\overline{f}: Y \rightarrow Z$ such that $f=\overline{f} \circ \eta$. A subset $A$ of a $T_{0}$ space $X$ is called a well-filtered determined set (resp., {$\omega$-well-filtered determined set}), $\WD$-set (resp., $\WD_{\omega}$-set) for short, if for any continuous mapping $f : X\rightarrow Y$ to an well-filtered (resp., $\omega$-well-filtered) space $Y$, there exists a unique $y_{A} \in Y$ such that $\cl(f(A)) = \cl(\{y_{A}\})$. The set of all closed $\WD$-sets (resp., $\WD_\omega$-sets) of $X$ is denoted by $\mathbf{WD}(X)$ (resp., $\mathbf{WD_{\omega}}(X)$). 
A standard construction for the well-filterification (resp., $\omega$-well-filterification) of a $T_{0}$ space $X$ is to set
$\mathbf{WD}(X)$ (resp., $\mathbf{WD}_{\omega}(X)$) topologized by the open sets $U^{w}:=\{ A \in \mathbf{WD}(X) : A \cap U \neq \emptyset\}$ (resp., $U^{w}:=\{ A \in \mathbf{WD_\omega}(X) : A \cap U \neq \emptyset\}$) for each open subset $U$ of $X$. See for example \cite[Theorem~6.8]{XiaoquanXu}. We write the resulting space as $X^w$ (resp., $X^{\omega-w}$). 

\subsection{$d$-spaces}
Recall that a monotone convergence space $Y$ together with a topological embedding $j:X \rightarrow Y$ with $j(X)$ a $d$-dense subset of $Y$ is called a $D$-completion of the $T_{0}$ space $X$, where a $d$-dense subset means to be dense in the $d$-topology of $Y$. A subset $A$ of a $T_{0}$ space $X$ is called \emph{tapered} if for each continuous function $f : X \rightarrow  Y$ mapping into a monotone convergence space $Y$, $ \sup f(A)$ always exists in $Y$.
A standard construction for the $D$-completion of a $T_{0}$ space $X$ is to set
	$$X^{d} := \{ A\subseteq X : A ~\mathrm{is} ~\mathrm{closed} ~\mathrm{and} ~\mathrm{tapered}\}$$
topologized by open sets $U^{ d}:=\{ A \in X ^{d} : A \cap U \neq \emptyset\}$ for each open subset $U$ of $X$. If we define $\eta^{d}_{X} :X \rightarrow X ^{d}$ by $\eta^{d}_{X} ( x ) = \cl(\{ x \}) $, then we obtain a $D$-completion \cite[Theorem 3.10]{zhongxizhang}, which we call the standard $D$-completion.

\begin{remark}
In a topological space $X$, all tapered sets are $\WD$ sets, $\WD$ sets are $\WD_\omega$ sets, and $\WD_\omega$ are irreducible sets. If $X$ is second countable, then they all coincide. 
\end{remark}

\subsection{Rudin spaces}

Let $X$ be a $T_{0}$ space. A nonempty subset $A$ of $X$ is said to have the \emph{compactly filtered property} or $\KF$ property (resp., countable compactly filtered property or $\KF_\omega$ property), if there exists a filtered family $\mathcal{K}$ of $Q(X)$ such that $\cl(A)$ is a minimal closed set that intersects all members of $\mathcal{K}$, where $\mathcal{K}$ is a (resp., countable) subset of $Q(X)$. We call such a set a $\KF$-set (resp., a $\KF_{\omega}$-set). Denote by $\mathbf{KF}(X)$ (resp., $\mathbf{KF_{\omega}}(X)$) the set of all closed $\KF$ (resp., closed $\KF_{\omega}$) subsets of $X$. All KF-sets are closed irreducible sets, and a space in which irreducible sets are KF-sets is called a \emph{Rudin} space. Second countable spaces are instances of Rudin spaces~\cite{XiaoquanXu}.

\section{The forbidden subspaces}

In this section, we give one of the main results of this paper and show that a $T_0$ topological space is $\omega$-well-filtered if and only if it does not contain certain subsets as subspaces. 
To start with, we have the two canonical examples of countably based $T_{0}$-spaces that are not $\omega$-well-filtered:

\begin{example}~
\begin{itemize}
\item  The space $S_1$ is defined as the set of all natural numbers $\mathbb{N}$ with the cofinite topology, in which proper closed subsets are finite subsets of $S_1$. 
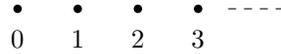
\begin{figure}[htb]
	\centering
	\begin{tikzpicture}[scale=0.8]
		\foreach \x in {1,2,3,4} \fill (\x,0) circle (2pt);
		\node (2) at(1,-0.5) {\small  $0$};
		\node (1) at(2,-0.5) {\small  $1$};
		\node (3) at(3,-0.5) {\small  $2$};
			\node (3) at(4,-0.5) {\small  $3$};
		\draw[dashed] (4.5,0)--(5.5,0);
	\end{tikzpicture}
	\caption{The non-$\omega$-well-filtered space $S_1$}
\end{figure}
\item  The space $S_D$ is defined as the set of natural numbers $\mathbb{N}$ with the Scott topology under the usual ordering.
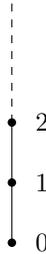
\begin{figure}[htb]
	\centering
	\begin{tikzpicture}[scale=0.8]
		\foreach \x in {0,1,2} \fill (2,\x) circle (2pt);
		\draw (2,0)--(2,2);
		\draw[dashed] (2,2)--(2,4);
		\node (2) at(2.5,1) {\small  $1$};
		\node (1) at(2.5,0) {\small  $0$};
		\node (3) at(2.5,2) {\small  $2$};
	\end{tikzpicture}
	\caption{The non-$\omega$-well-filtered space $S_D$}
\end{figure}
\end{itemize}
In both $S_1$ and $S_D$, $K_n = \{n, n+1, n+2, \cdots \}, n\in \mathbb N$ are nonempty compact saturated subsets, but the intersection of 
$K_n, n\in \mathbb N$ is empty. So neither of them is $\omega$-well-filtered. 

\end{example}

Recall that a subset of a space is \emph{locally-closed} if it is equal to the intersection of an open set with a closed set. A \emph{$T_{D}$-space} is a space in which every singleton subset is locally closed. A space is \emph{perfect} if and only if every non-empty open subset is infinite. In \cite[Theorem 3.5]{Brecht}, it was shown that a non-empty countably based perfect $T_{D}$-space with countably many points contains a perfect $\Pi_{2}^{0}$-subspace homeomorphic to either $\mathbb{Q}$, $S_1$ or $S_D$, where $\mathbb{Q}$ is the set of rational numbers with the subspace topology inherited from the space of reals. In \cite{Brecht}, de Brecht obtained the result that a countably based $T_0$-space is sober if and only if it does not contain a $\Pi_{2}^{0}$-subspace homeomorphic to $S_D$ or $S_1$. 

In the same spirit, we give a characterization for $\omega$-well-filtered spaces (which may fail to be countably based)  by two types of forbidden subspaces. Using this characterization, we present a new description for the standard $\omega$-well-filterification for first-countable $T_0$ spaces and give an answer to Problem~\ref{rudin} mentioned in the Introduction. To achieve this, we start with the following lemma.

\begin{lemma}\label{KF}
Let $X$ be a $T_0$ space. If $A\in \mathbf{KF}_{\omega}(X)$ and $A\notin \{\da _X x\mid x\in X\}$, then there is a  closed subset $B$ in the strong topology of $X$ that is homeomorphic to $S_D$ or $S_1$ such that $\cl_X(B)=A$.
\end{lemma}
\begin{proof}
Let $A\in \mathbf{KF}_{\omega}(X)$. Then there exists a filtered family $\mathcal{K}$ of $Q(X)$ such that $A$ is a minimal closed set that intersects all members of $\mathcal{K}$, where $\mathcal{K}$ is a countable subset of $Q(X)$. The countability of $\mathcal{K}$ guarantees that we can assume that $\mathcal{K}=(K_{n})_{n\in \mathbb{N}}$ is a descending chain, that is, $K_{n+1}\subseteq K_n$ for any $n\in \mathbb{N}$. Then we pick $x_n\in A\cap K_{n}$ for any $n\in \mathbb{N}$. Set $H=\{x_n\mid n\in \mathbb{N}\}$ and note that $\cl_X (H)$ intersects all members of $\mathcal{K}$. By the minimality of $A$, we have $A=\cl_X (H)$. Now we prove that $H$ is an infinite set. We proceed by contradiction. Suppose $H$ is finite. Then $\cl_X (H)=A=\bigcup_{x\in H}\da_X x$. The irreducibility of $A$ implies that $A=\da x$ for some $x\in H$. This violates the assumption that $A\notin \{\da_X x\mid x\in X\}$. Therefore, $H$ is an infinite set. It indicates that $\cl_X(E)=\cl_X(H)=A$ for any infinite subset $E$ of $H$.

Claim $1$: $H\backslash U$ is finite for any $U\in \mathcal{O}(X)$ with $U\cap H\neq \emptyset$. 

Let $U\in \mathcal{O}(X)$ with $U\cap H\neq \emptyset$. Assume that $H\backslash U$ is infinite. Since $H\subseteq A$, $(A\backslash U)\cap H$ is infinite. It means that $(A\backslash U)\cap K_n\neq \emptyset$ for any $n\in \mathbb{N}$. Applying the minimality of $A$ again, we know that $A=A\backslash U$. It follows that $H\subseteq A\subseteq X\backslash U$, which contradicts the assumption that $H\cap U\neq \emptyset$.

Claim $2$: $H$ is a non-empty countably based perfect $T_D$-space with countably many points.

One can see directly that $H$ is a non-empty countably based perfect space with countably many points from Claim $1$. It remains to confirm that $H$ is a $T_D$ space. To this end, let $x\in H$, we show that $\da_H x\backslash\{x\}$ is finite. If $\da_H x\backslash \{x\}$ is infinite, then $\cl_X(\da _H x\backslash \{x\})=\cl_X(H)=A=\da _X x$. A contradiction. Hence, $\da_H x\backslash \{x\}$ is finite. This means that $\da_H (\da_H x\backslash \{x\})=\da _H x\backslash \{x\}$ is a closed set of $H$, which is equivalent to saying that $H$ is a $T_D$ space.

On account of \cite[Theorem 3.5]{Brecht}, we know that there is a perfect $\Pi_{2}^{0}$-subspace $B$ of $H$ homeomorphic to either $\mathbb{Q}$, $S_1$ or $S_D$. It manifests that $B$ is a Skula closed subset of $H$.
Using Claim~$1$, we know that $B$ must be $S_1$ or $S_D$. The fact that $B$ is an infinite subset of $H$ infers that $\cl_X(B)=A$. We proceed  to check that $H$, hence $B$,  is a Skula closed subset of $X$.

Write $H_n=\{x_i\mid i\leq n\}\cup\bigcap_{i\leq n}K_{i}$ for any $n\in \mathbb{N}$. Because $K=\bigcap_{U\in \mathcal{U}}U$ for any $K\in Q(X)$, where $\mathcal{U}=\{U\in \mathcal{O}(X)\mid K\subseteq U\}$, the set $\bigcap_{i\leq n}K_{i}$ is Skula closed. The single set $\{x\}$ is Skula closed for any $x\in X$ owing to the fact that $\{x\}=\da_X x\cap \ua_X x$. It implies that $H_n$ is Skula closed for any $n\in \mathbb{N}$. The proof is complete if we know that $H=(\bigcap_{n\in \mathbb{N}}H_n)\cap A$. Indeed, that $H\subseteq (\bigcap_{n\in \mathbb{N}}H_n)\cap A$ is trivial. Now let $x\in( \bigcap_{n\in \mathbb{N}}H_n)\cap A$. Assume for the sake of a
contradiction that $x\notin H$. Then $x\in \bigcap_{n\in \mathbb{N}}K_{n}\cap A$. The minimality of $A$ tells us that $A=\da_X x$, a contradiction. So $H$ is indeed Skula closed. 
\end{proof}

By applying the above lemma, we have the following theorem.
\begin{theorem}
Let $X$ be a $T_0$ space. Then $X$ is $\omega$-well-filtered if and only if it does not contain a copy of $S_D$ or $S_1$ as Skula closed subsets.
\end{theorem}
\begin{proof}
If $X$ is an $\omega$-well-filtered space, then we claim that $Y$ is $\omega$-well-filtered for any Skula closed subset $Y$ of $X$. To this end, let $A\in \mathbf{KF}_{\omega}(Y)$. Then $A\in \mathbf{KF}_{\omega}(X)$. From \cite{XiaoquanXu}, we know that $X$ is $\omega$-well-filtered if and only if $A$ is the closure of some single point $x\in X$ for any $A\in \mathbf{KF}_{\omega}(X)$. This yields that $\cl_{X}(A)=\da_X a$ for some $a\in X$. Now it suffices to show that $a\in Y$. For any $U\in \mathcal{O}(X)$ with $a\in U$, then $\emptyset\neq U\cap A\subseteq U\cap A\cap \da_X a\subseteq U\cap Y\cap \da _X a$, which says that $a$ is contained in the Skula closure of $Y$. Hence, $a$ is in $Y$ from the assumption that $Y$ is Skula closed. This implies that $X$ does not contain a Skula closed subspace homeomorphic to $S_D$ or $S_1$.

For the converse, we proceed by contradiction. Suppose that $X$ is not $\omega$-well-filtered. Then we can find $A\in \mathbf{KF}_{\omega}(X)$ satisfying $A\notin \{\da_X x\mid x\in X\}$. According to Lemma \ref{KF}, we can conclude that there exists a subset $B$ of $A$ which is closed in the strong topology of $X$ and homeomorphic to $S_D$ or $S_1$ such that $\cl_X(B)=A$. A contradiction. 
\end{proof}

When the underlying $T_0$ space is first-countable, we could say more and Lemma~\ref{KF} can be sharpened into the following:

\begin{theorem}
Let $X$ be a first-countable $T_0$ space. If $A\in \mathbf{WD}_{\omega}(X)$ and $A\notin \{\da_X x\mid x\in X\}$, then there is a $\Pi^{0}_2$-subset $B$ of $A$ homeomorphic to $S_D$ or $S_1$ and $\cl_X(B)=A$.
\end{theorem}
\begin{proof}
Let $A\in \mathbf{WD}_{\omega}(X)$ with $A\notin \{\da_X x\mid x\in X\}$. Consider the canonical embedding $\eta_{X}^{\omega-w} \colon X\to X^{\omega-w}: x\mapsto \cl{\{x\}}$. Then $\eta_{X}^{\omega-w}(A)$ is a $\WD_{\omega}$-set of $X^{\omega-w}$. The $\omega$-well-filteredness of $X^{\omega-w}$ induces $\cl_{X^{\omega-w}}(\eta_{X}^{\omega-w}(A))=\da_{X^{\omega-w}} A$. From \cite[Theorem 4.15]{Miao3}, we conclude that $X^{\omega-w}$ is first countable. It follows that there is a countable descending neighbourhood basis $(\mathcal{U}_n)_{n\in \mathbb{N}}$ of $A$. This means that $\mathcal{U}_{n}\cap \eta_{X}^{\omega-w}(A)\neq \emptyset$ for any $n\in \mathbb{N}$. Choose $\eta_{X}^{\omega-w}(a_n)\in \mathcal{U}_{n}\cap \eta_{X}^{\omega-w}(A)$ for any $n\in \mathbb{N}$. Write $\mathcal{B}=\{\eta_{X}^{\omega-w}(a_n)\mid n\in \mathbb{N}\}$. One sees immediately that $\cl_{X^{\omega-w}}(\mathcal{B})=\da _{X^{\omega-w}}A$.

Next, we show that $\cl_X(\{a_n\mid n\in \mathbb{N}\})=A$. 

That $\cl_X(\{a_n\mid n\in \mathbb{N}\})\subseteq A$ is trivial.  Now assume $a\in A$. For any $a\in U$, we have $A\in U^{\omega-w}$. It turns out that $U^{\omega-w} \cap \mathcal{B}\neq \emptyset$. This suggests that $U\cap \{a_n\mid n\in \mathbb{N}\}\neq \emptyset$. Hence, $\cl_X(\{a_n\mid n\in \mathbb{N}\})=A$. Set $H=\{a_n\mid n\in \mathbb{N}\}$. Note that the assumption that $A\notin\{\da _X x\mid x\in X\}$ implies that $H$ is an infinite subset, which means that $\cl_X(B)=\cl_X(H)=A$ for any infinite subset $B$ of $H$. 

Since $X$ is first countable, we can assume that $(U_n(a_k))_{n\in \mathbb{N}}$ is a countable descending neighbourhood basis of $a_k$ for each $k\in \mathbb{N}$. We inductively define an infinite sequence $\{x_{n} \} _{n\in \mathbb{N}}$ of distinct elements of $H$. Write $U_{0}=V_{0}=U_0(a_0)$ and let $x_{0}=a_0$. 
Assume that we have defined a sequence $x_{0},\cdots,x_{n}\in H$ and open sets $U_{0},\cdots, U_{n},V_{0}$, $\cdots, V_{n}$ with $V_{i}\subseteq U_{i},V_{i}\cap H\neq \emptyset$ for any $i\leq n$. We choose $U_{n+1},V_{n+1}$ as follows. We define
\begin{center}
	$U_{n+1}=V_{n}\backslash \bigcup_{i\leq n}\da x_{i}$;
	
	$V_{n+1}=U_{n+1}\cap\bigcap_{i\leq n}U_{n+1}(x_i)$.
\end{center}
Then $V_{n+1}\cap H\neq \emptyset$ as $H$ is irreducible. Pick $x_{n+1}$ to be any element of $V_{n+1}\cap H$. Thus, $S=\{x_{n}\mid n\in \mathbb{N}\}$ is an infinite subset of $H$. For a similar proof of Theorem 4.3 in \cite{Brecht}, $S$ is a countable non-Hausdorff  perfect $T_{D}$ space. Again by applying Theorem 3.5 in \cite{Brecht},  there is $B\in \Pi^{ 0}_{2}(S)$ homeomorphic to either $S_D$ or $S_1$. The infiniteness of $B$ infers that $\cl_X(B)=\cl_X(H)=A$. Now it remains to show that $S\in \Pi^{0}_2(X)$.

Write $S_n=\{x_i\mid i\leq n\}\cup \bigcap_{i\leq n}V_i$ for any $n\in \mathbb{N}$. It follows straightforward that $S_n\in \Pi^0_2(X)$ from the first-countability of $X$. So it suffices to check that $S=(\bigcap_{n\in \mathbb{N}}S_n)\cap \cl_X(S)$. Let $x\in (\bigcap_{n\in \mathbb{N}}S_n)\cap \cl_X(S)$. Assume for the sake of a contradiction that $x\notin S$. It means that $x\in \bigcap_{n\in \mathbb{N}}V_n$. Note that for any $x_i\in S$, $\bigcap_{n\in \mathbb{N}}V_n\subseteq \bigcap_{n\in \mathbb{N}}U_{n}(x_i)=\ua_X x_i$. This yields that $\cl_X(S)=\da_X x=\cl_X(H)=A$, which violates the fact that $A\notin\{\da_X x\mid x\in X\}$. The other direction of inclusion is obvious. 
\end{proof}

Now the following result is a straightforward corollary to the above theorem, and it provides a more refined description for the $\omega$-well-filtered spaces completion. 

\begin{corollary}\label{C}
Let $X$ be a first-countable $T_0$ space. Then $X^{\omega-w}=\{\da_X x\mid x\in X\}\cup \{A\in \Gamma(X)\mid A=\cl_X(B), B\cong S_D\}\cup \{A\in \Gamma(X)\mid A=\cl_X(B), B\cong S_1\}$.
\end{corollary}

As for second countable $T_0$ spaces, $\omega$-well-filterification coincides with sobrification, this description also applies to sobrification. 
\begin{corollary}
Let $X$ be a second countable $T_0$ space. Then $X^s=X^{\omega-w}=\{\da_X x\mid x\in X\}\cup \{A\in \Gamma(X)\mid, A=\cl_X(B), B\cong S_D\}\cup \{A\in \Gamma(X)\mid A=\cl_X(B), B\cong S_1\}$.
\end{corollary}

Moreover, note that $S_1$ is not a tapered closed set. Hence we have the following result for $d$-completions.
\begin{corollary}
Let $X$ be a second countable $T_0$ space. Then $X^d=\{\da_X x\mid x\in X\}\cup \{A\in \Gamma(X)\mid A=\cl_X(B), B\cong S_D\}$.
\end{corollary}

\section{An answer to Xu's question}

Recall that a Rudin space is a space in which all closed irreducible sets are KF-sets.
We now give an answer to Problem \ref{rudin} proposed by Xu et al. in \cite{XiaoquanXu2}, as promised in the Introduction. Namely, we will construct a first-countable $T_0$ space that is not a Rudin space. In other words, we find a first-countable $T_0$ space in which not all closed irreducible sets are KF-sets. 

%

\begin{figure}[htb]
	\centering
	\begin{tikzpicture}[scale=0.8]
		\foreach \x in {0,1,2} \fill (0,\x) circle (2pt);
		\draw (0,0)--(0,2);
		\draw[dashed] (0,2)--(0,4);
		\node (2) at(0.5,1) {\tiny  $(0,1)$};
		\node (1) at(0.5,0) {\tiny $(0,0)$};
		\node (3) at(0.5,2) {\tiny  $(0,2)$};
		\foreach \x in {0,1,2} \fill (1,\x) circle (2pt);
		\draw (1,0)--(1,2);
		\draw[dashed] (1,2)--(1,4);
		\node (2) at(1.5,1) {\tiny $(1,1)$};
		\node (1) at(1.5,0) {\tiny  $(1,0)$};
		\node (3) at(1.5,2) {\tiny  $(1,2)$};
		\foreach \x in {0,1,2} \fill (2,\x) circle (2pt);
		\draw (2,0)--(2,2);
		\draw[dashed] (2,2)--(2,4);
		\node (2) at(2.5,1) {\tiny $(2,1)$};
		\node (1) at(2.5,0) {\tiny  $(2,0)$};
		\node (3) at(2.5,2) {\tiny  $(2,2)$};
		\draw[dashed] (3,0)--(5,0);
		\draw[dashed] (3,1)--(5,1);
		\draw[dashed] (3,2)--(5,2);
		\foreach \x in {0,1,2} \fill (5.5,\x) circle (2pt);
		\draw (5.5,0)--(5.5,2);
		\draw[dashed] (5.5,2)--(5.5,4);
		\node (2) at(6,1) {\tiny $(\alpha,1)$};
		\node (1) at(6,0) {\tiny  $(\alpha,0)$};
		\node (3) at(6,2) {\tiny  $(\alpha,2)$};
		\foreach \x in {0,1,2} \fill (6.5,\x) circle (2pt);
		\draw (6.5,0)--(6.5,2);
		\draw[dashed] (6.5,2)--(6.5,4);
		\node (2) at(7.2,1) {\tiny $(\alpha+1,1)$};
		\node (1) at(7.2,0) {\tiny  $(\alpha+1,0)$};
		\node (3) at(7.2,2) {\tiny  $(\alpha+1,2)$};
		\foreach \x in {0,1,2} \fill (8,\x) circle (2pt);
		\draw (8,0)--(8,2);
		\draw[dashed] (8,2)--(8,4);
		\node (2) at(8.7,1) {\tiny $(\alpha+2,1)$};
		\node (1) at(8.7,0) {\tiny  $(\alpha+2,0)$};
		\node (3) at(8.7,2) {\tiny  $(\alpha+2,2)$};
		\draw[dashed] (10,0)--(12,0);
		\draw[dashed] (10,1)--(12,1);
		\draw[dashed] (10,2)--(12,2);
\end{tikzpicture}
	\caption{The space $X=[0,\omega_1)\times \mathbb{N}$ in Example \ref{open}}
\end{figure}
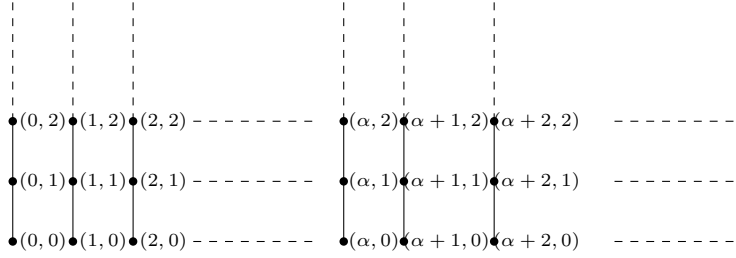
\begin{figure}[htb]
	\centering
	\begin{tikzpicture}[scale=0.8]
		\foreach \x in {0,1,2,4} \fill (0,\x) circle (2pt);
		\draw (0,0)--(0,2);
		\draw[dashed] (0,2)--(0,4);
		\node (2) at(0.5,1) {\tiny  $\da (0,1)$};
		\node (1) at(0.5,0) {\tiny $\da (0,0)$};
		\node (3) at(0.5,2) {\tiny  $\da(0,2)$};
		\node (4) [above] at(0,4) {\tiny  $\{0\}\times \mathbb{N}$};
		\foreach \x in {0,1,2,4.5} \fill (1,\x) circle (2pt);
		\draw (1,0)--(1,2);
		\draw[dashed] (1,2)--(1,4.5);
		\node (2) at(1.5,1) {\tiny $\da(1,1)$};
		\node (1) at(1.5,0) {\tiny  $\da(1,0)$};
		\node (3) at(1.5,2) {\tiny  $\da(1,2)$};
		\node (4) [above] at(1,4.5) {\tiny  $\{0,1\}\times \mathbb{N}$};
		\foreach \x in {0,1,2,5} \fill (2,\x) circle (2pt);
		\draw (2,0)--(2,2);
		\draw[dashed] (2,2)--(2,5);
		\node (2) at(2.5,1) {\tiny $\da(2,1)$};
		\node (1) at(2.5,0) {\tiny  $\da(2,0)$};
		\node (3) at(2.5,2) {\tiny  $\da(2,2)$};
		\node (4) [above] at(2,5) {\tiny  $\{0,1,2\}\times \mathbb{N}$};
		\draw[dashed] (3,0)--(5,0);
		\draw[dashed] (3,1)--(5,1);
		\draw[dashed] (3,2)--(5,2);
		\foreach \x in {0,1,2,6.5} \fill (5.5,\x) circle (2pt);
		\draw (5.5,0)--(5.5,2);
		\draw[dashed] (5.5,2)--(5.5,6.5);
		\node (2) at(6,1) {\tiny $\da(\alpha,1)$};
		\node (1) at(6,0) {\tiny  $\da(\alpha,0)$};
		\node (3) at(6,2) {\tiny  $\da(\alpha,2)$};
		\node (4) [above] at(5.5,6.5) {\tiny  $[0, \alpha]\times \mathbb{N}$};
		\foreach \x in {0,1,2,7} \fill (6.5,\x) circle (2pt);
		\draw (6.5,0)--(6.5,2);
		\draw[dashed] (6.5,2)--(6.5,7);
		\node (2) at(7.3,1) {\tiny $\da(\alpha+1,1)$};
		\node (1) at(7.3,0) {\tiny  $\da(\alpha+1,0)$};
		\node (3) at(7.3,2) {\tiny  $\da(\alpha+1,2)$};
		\node (4) [above] at(6.5,7) {\tiny  $[0, \alpha+1]\times \mathbb{N}$};
		\foreach \x in {0,1,2,7.7} \fill (8.1,\x) circle (2pt);
		\draw (8.1,0)--(8.1,2);
		\draw[dashed] (8.1,2)--(8.1,7.7);
		\node (2) at(8.9,1) {\tiny $\da(\alpha+2,1)$};
		\node (1) at(8.9,0) {\tiny  $\da(\alpha+2,0)$};
		\node (3) at(8.9,2) {\tiny  $\da(\alpha+2,2)$};
		\node (4) [above] at(8.1,7.7) {\tiny  $[0, \alpha+2]\times \mathbb{N}$};
		\draw[dashed] (10,0)--(12,0);
		\draw[dashed] (10,1)--(12,1);
		\draw[dashed] (10,2)--(12,2);
		\draw (0,4)--(1,4.5);
		\draw (1,4.5)--(2,5);
		\draw[dashed] (2,5)--(5.5,6.5);
		\draw(5.5,6.5)--(6.5,7);
		\draw(6.5,7)--(8.1,7.7);
			\draw[dashed] (8.2,7.7)--(11.5,8.5);
	\end{tikzpicture}
	\caption{The standard $\omega$-well-filterification of $X$ in Example \ref{open}}
\end{figure}
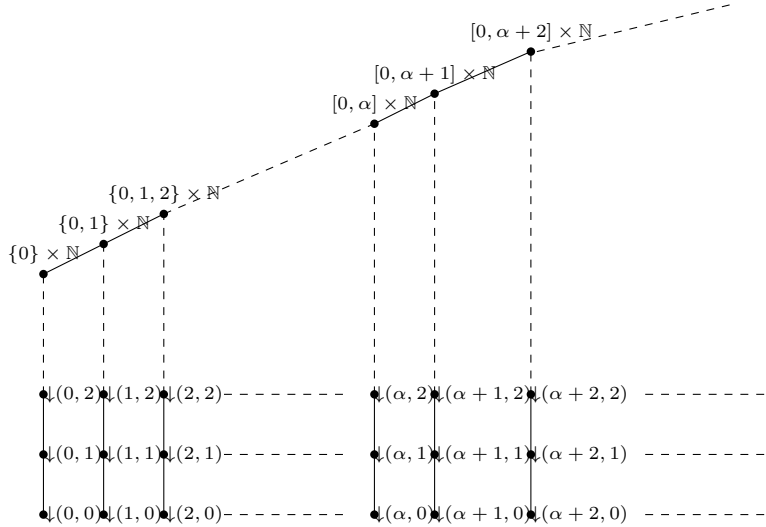
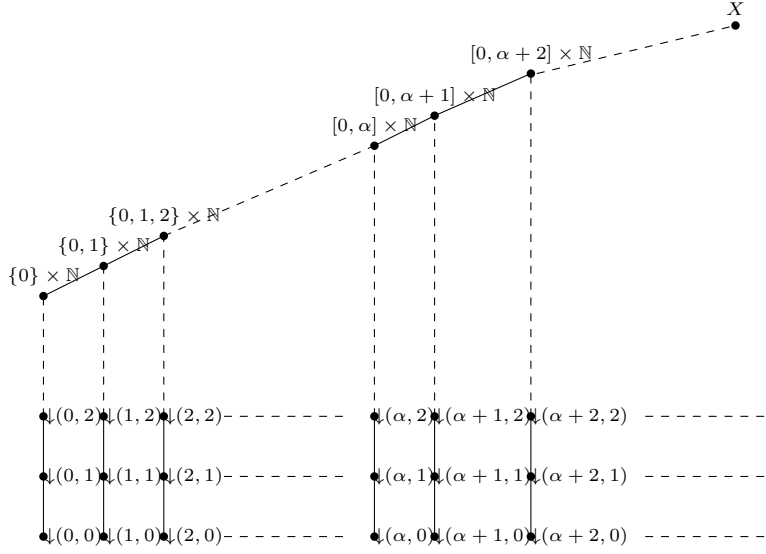
\begin{figure}[htb]
	\centering
	\begin{tikzpicture}[scale=0.8]
		\foreach \x in {0,1,2,4} \fill (0,\x) circle (2pt);
		\draw (0,0)--(0,2);
		\draw[dashed] (0,2)--(0,4);
		\node (2) at(0.5,1) {\tiny  $\da (0,1)$};
		\node (1) at(0.5,0) {\tiny $\da (0,0)$};
		\node (3) at(0.5,2) {\tiny  $\da(0,2)$};
		\node (4) [above] at(0,4) {\tiny  $\{0\}\times \mathbb{N}$};
		\foreach \x in {0,1,2,4.5} \fill (1,\x) circle (2pt);
		\draw (1,0)--(1,2);
		\draw[dashed] (1,2)--(1,4.5);
		\node (2) at(1.5,1) {\tiny $\da(1,1)$};
		\node (1) at(1.5,0) {\tiny  $\da(1,0)$};
		\node (3) at(1.5,2) {\tiny  $\da(1,2)$};
		\node (4) [above] at(1,4.5) {\tiny  $\{0,1\}\times \mathbb{N}$};
		\foreach \x in {0,1,2,5} \fill (2,\x) circle (2pt);
		\draw (2,0)--(2,2);
		\draw[dashed] (2,2)--(2,5);
		\node (2) at(2.5,1) {\tiny $\da(2,1)$};
		\node (1) at(2.5,0) {\tiny  $\da(2,0)$};
		\node (3) at(2.5,2) {\tiny  $\da(2,2)$};
		\node (4) [above] at(2,5) {\tiny  $\{0,1,2\}\times \mathbb{N}$};
		\draw[dashed] (3,0)--(5,0);
		\draw[dashed] (3,1)--(5,1);
		\draw[dashed] (3,2)--(5,2);
		\foreach \x in {0,1,2,6.5} \fill (5.5,\x) circle (2pt);
		\draw (5.5,0)--(5.5,2);
		\draw[dashed] (5.5,2)--(5.5,6.5);
		\node (2) at(6,1) {\tiny $\da(\alpha,1)$};
		\node (1) at(6,0) {\tiny  $\da(\alpha,0)$};
		\node (3) at(6,2) {\tiny  $\da(\alpha,2)$};
		\node (4) [above] at(5.5,6.5) {\tiny  $[0, \alpha]\times \mathbb{N}$};
		\foreach \x in {0,1,2,7} \fill (6.5,\x) circle (2pt);
		\draw (6.5,0)--(6.5,2);
		\draw[dashed] (6.5,2)--(6.5,7);
		\node (2) at(7.3,1) {\tiny $\da(\alpha+1,1)$};
		\node (1) at(7.3,0) {\tiny  $\da(\alpha+1,0)$};
		\node (3) at(7.3,2) {\tiny  $\da(\alpha+1,2)$};
		\node (4) [above] at(6.5,7) {\tiny  $[0, \alpha+1]\times \mathbb{N}$};
		\foreach \x in {0,1,2,7.7} \fill (8.1,\x) circle (2pt);
		\draw (8.1,0)--(8.1,2);
		\draw[dashed] (8.1,2)--(8.1,7.7);
		\node (2) at(8.9,1) {\tiny $\da(\alpha+2,1)$};
		\node (1) at(8.9,0) {\tiny  $\da(\alpha+2,0)$};
		\node (3) at(8.9,2) {\tiny  $\da(\alpha+2,2)$};
		\node (4) [above] at(8.1,7.7) {\tiny  $[0, \alpha+2]\times \mathbb{N}$};
		\draw[dashed] (10,0)--(12,0);
		\draw[dashed] (10,1)--(12,1);
		\draw[dashed] (10,2)--(12,2);
		\draw (0,4)--(1,4.5);
		\draw (1,4.5)--(2,5);
		\draw[dashed] (2,5)--(5.5,6.5);
		\draw(5.5,6.5)--(6.5,7);
		\draw(6.5,7)--(8.1,7.7);
		\draw[dashed] (8.2,7.7)--(11.5,8.5);
		\fill (11.5,8.5) circle (2pt);
		\node (5) [above] at (11.5,8.5) {\tiny  $X$};
	\end{tikzpicture}
	\caption{The standard well-filterification and sobrification of $X$ in Example \ref{open}}\label{f1}
\end{figure}

\begin{example}\label{open}
Let $X=[0,\omega_1)\times \mathbb{N}$, where $\omega_1$ denotes the first uncountable ordinal number and $\mathbb{N}$ the set of natural numbers. 
For $\alpha \in [0,\omega_1)$, we use $\ua \alpha$ to denote the set of all ordinals that are larger than or equal to $\alpha$; and for $n\in \mathbb N$, $\ua n$ is the set of 
all natural numbers that are larger than or equal to $n$. For a subset $U$ of $X$, we define 
$U_{[0,\omega_1)}=\{\alpha\in [0,\omega_1)\mid (\{\alpha\}\times \mathbb{N})\cap U\neq \emptyset\}$, and $\mathcal{B}$ consists of all subsets $U$ of $X$ such that

\begin{itemize}
\item $U_{[0,\omega_1)}=\ua \alpha_{0}$ for some $\alpha_{0} \in [0,\omega_1)$;
\item $U\cap (\{\alpha_0\}\times \mathbb{N}) =  \{\alpha_0\} \times \ua n $ for some $n\in \mathbb N$;
\item for $\beta >\alpha_0$, $U\cap (\{\beta\}\times \mathbb{N})= \{\beta\} \times \ua n_0$ for some fixed $n_0\in \mathbb N$.
\end{itemize}
Then $\mathcal{B}$ is a basis. Let $\tau$ be the topology generated by $\mathcal{B}$. Then $(X,\tau)$ is first-countable, and $X\in \mathbf{IRR}(X)\backslash \mathbf{WF}(X)$.
\end{example}
\begin{proof}
Our proofs go as follows. 
\begin{enumerate}
\item $\mathcal{B}$ is a basis.\\
It is easy to see that $X\in \mathcal{B}$. We prove that $\mathcal{B}$ is closed for finite intersections. To this end, let $U_1,U_2\in \mathcal{B}$. Then there are $\alpha_1,\alpha_2\in [0,\omega_1)$ and $n_0,n_1\in \mathbb{N}$ such that $(U_i)_{[0,\omega_1)}=\ua \alpha_i$, $U_i\cap(\{\beta\}\times \mathbb{N})=\ua n_i$ for any $\beta >\alpha_i$, $i\in\{1,2\}$. Set $\alpha_0=\max\{\alpha_1,\alpha_2\}$ and $n_0=\max\{n_1,n_2\}$. Then $(U_1\cap U_2)_{[0,\omega_1)}=\ua \alpha_0$, and for any $\beta>\alpha_0$, we have that $(U_1\cap U_2)\cap (\{\beta\}\times \mathbb{N})=\ua n_0$, and $(U_1\cap U_2)\cap (\{\alpha_0\}\times \mathbb{N})$ is of course of the form $\ua k$ for some $k\in \mathbb{N}$. This means that $(U_1\cap U_2)\in \mathcal{B}$.

\item $(X,\tau)$ is first-countable.\\
For any $(\alpha,n)\in X$, set $U_{k}=\{\alpha\}\times \ua n\cup \bigcup_{\beta>\alpha}\{\beta\}\times \ua k$ for $k\in \mathbb{N}$. Then $U_k$ is open in $(X,\tau)$ for each $k\in \mathbb N$ and $(\alpha,n)\in U_k$. Next, we show that $(U_k)_{k\in \mathbb{N}}$ form a neighbourhood basis of $(\alpha,n)$. Let $U\in \mathcal{B}$ with $(\alpha,n)\in \mathcal{B}$. Then there exists $\alpha_0\in [0,\omega_1)$, $n_{0}\in \mathbb{N}$ such that $U_{[0,\omega_1)}=\ua \alpha_0$, $U\cap (\{\beta\}\times \mathbb{N})=\ua n_0$ for $\beta > \alpha_0$. This means that $(\alpha,n)\in U_{n_0+1}\subseteq U$. Hence, $(X,\tau)$ is first-countable.

\item $K$ is a countable set for each $K\in Q(X)$.\\
Suppose that there exists an uncountable compact saturated subset $K$ of $X$. Set $L_n=[0,\omega_1)\times \{n\}$ for any $n\in \mathbb{N}$. Then  $X=\bigcup_{n\in \mathbb{N}}L_n$. This means that $K=\bigcup_{n\in \mathbb{N}}(K\cap L_n)$. The assumption that $K$ is uncountable implies that for some $n\in \mathbb{N}$, $K\cap L_n$ is uncountable. Pick a sequence $(\alpha_k,n)_{k\in \mathbb{N}}$ inside $K\cap L_n$. The compactness guarantees that there exists a cluster point $(\beta_0,n_0)$ of the sequence $(\alpha_k,n)_{k\in \mathbb{N}}$. Note that $(\beta_0,n_0)\in V$, where $V=\{\beta_0\}\times \ua n_0\cup \bigcup_{\beta>\beta_0}\{\beta\}\times \ua (n+1)$ is an open set. It follows that $V_{n+1}\cap \{(\alpha_k,n)\mid n\in \mathbb{N}\}$ is infinite. But that cannot be true by definition. 

\item $X\in \mathbf{IRR}(X)\backslash \mathbf{WF}(X)$. \\
One easily verifies that any two non-empty subsets of $X$ actually intersect, so that $X\in \mathbf{IRR}(X)$. 
We now show that $X\notin \mathbf{KF}(X)$ and proceed by contradiction. Suppose that $X\in \mathbf{KF}(X)$. Then there exists a filtered family $(K_i)_{i\in I}$ of $Q(X)$ such that $X$ is a minimal closed set that intersects all members of $\mathcal{K}$. Fix $i_0\in I$, then $K_{i_0}$ is countable by~$(3)$, and then we can find $\alpha\in [0,\omega_1)$ such that $K_{i_0}\subseteq [0,\alpha]\times \mathbb{N}$. Let $B_{\alpha}=[0,\alpha]\times \mathbb{N}$. Obviously, $B_{\alpha}$ is a closed set of $X$ and $B_{\alpha}$ intersects all members of $\mathcal{K}$. This would have implied that $X=B_{\alpha}$ by minimality of $X$, which is impossible.
\end{enumerate}
\end{proof}


\section{A topological characterization for $\omega$-well-filterification}
In \cite{GHKLMS}, Gierz et al. presented a characterization for the standard sobrification via the strong topology. In this section, we give a characterization for the standard $\omega$-well-filterification by the strong$_*$ topology that we introduce below. 
\begin{lemma}
Let $X$ be a $T_0$ space and $$\mathcal{B}=\{A\subseteq X\mid \text{for all}~B\in \KF_{\omega}(A)~\text{with}~\sup B~\text{existing}, \sup B\in A\}.$$ Then $\mathcal{B}$ consists of closed sets of a topology $\tau$ on $X$, which we call the \emph{strong$_*$ topology} of $X$.
\end{lemma}

\begin{proof}
Let $(A_i)_{i\in I}$ be a family of $\mathcal{B}$, $B\in \KF_{\omega}(\bigcap_{i\in I}A_i)$ with $\sup B$ existing. The fact that continuous mappings preserve $\KF_{\omega}$-sets means that $B\in \KF_{\omega}(A_i)$ for any $i\in I$. It turns out that $\sup B\in A_i$ because $A_i\in \mathcal{B}$ for each $i\in I$. By  the definition of $\mathcal{B}$, we have that $\bigcap_{i\in I}A_i\in \mathcal{B}$.

Let $A,B\in \mathcal{B}$, $E\in \KF_{\omega}(A\cup B)$ with $\sup E$ existing. Now it suffices to confirm that $\sup E\in A\cup B$. The case that $\cl_{A\cup B}(E)\in \{\da_{A\cup B} x\mid x\in A\cup B\}$ is trivial, so we assume that $\cl_{A\cup B}(E)\notin \{\da _{A\cup B}x\mid x\in A\cup B\}$. Then there exists a descending countable family $\mathcal{K}$ of $Q(A\cup B)$ such that $\cl_{A\cup B}(E)$ is a minimal closed set that intersects all members of $\mathcal{K}$. Set $\mathcal{K}=\{K_n\mid n\in \mathbb{N}\}$. Pick $x_n\in \cl_{A\cup B}(E)\cap K_n$ for any $n\in \mathbb{N}$. Define $H=\{x_n\mid n\in \mathbb{N}\}$. From the proof of Lemma~\ref{KF}, we could know that $H\backslash U$ is finite for every $U\in \mathcal{O}(X)$ with $U\cap H\neq \emptyset$ and $H$ is an infinite subset. It follows that $F$ is compact for every subset $F$ of $H$. Note that $H\subseteq \cl_{A\cup B}(E)\subseteq A\cup B$. This leads to the fact that $H\cap A$ is infinite or $H\cap B$ is infinite. Without loss of generality, we assume that $H\cap A$ is infinite, and then we would know that $\cl_{A\cup B}(H\cap A)=\cl_{A\cup B}(H)=\cl_{A\cup B}(E)$, which follows easily that $\sup E=\sup (H\cap A)$.
Next, we claim that $H\cap A$ is a $\KF_{\omega}$-set of $A$.
Let $H\cap A=\{y_n\mid n\in \mathbb{N}\}$ and $H_n=\{y_m\mid m\geq n\}$ for any $n\in \mathbb{N}$. Then $H_n$ is compact and $\cl_A(H\cap A)\cap \ua_A H_n\neq \emptyset$ for any $n\in \mathbb{N}$. Now we deduce that $\cl_A(H\cap A)$ is a minimal closed set in $A$ that intersects $\ua_A H_n$ for  $n\in \mathbb{N}$. Suppose $R\subseteq \cl_A(H\cap A)$ is closed in $A$ and $R\cap \ua_A H_n\neq \emptyset$ for any $n\in \mathbb{N}$. It follows that $R\cap H_n\neq \emptyset$ for each $n\in \mathbb{N}$ because $R$ is a lower set of $A$. This implies that $R\cap K_n\neq \emptyset$ for any $n\in \mathbb{N}$. Note that $R\subseteq \cl_A(H\cap A)\subseteq \cl_{A\cup B}(H)=\cl_{A\cup B}(E)$. Then $\cl_{A\cup B}(R)\subseteq \cl_{A\cup B}(E)$. The minimality of $\cl_{A\cup B}(E)$ implies that $\cl_{A\cup B}(R)=\cl_{A\cup B}(E)$. It turns out that $\cl_A(H\cap A)\subseteq \cl_{A\cup B}(H)\cap A=\cl_{A\cup B}(E)\cap A=\cl_{A\cup B}(R)\cap A=\cl_{A}(R)=R$. Hence, $H\cap A\in \KF_{\omega}(A)$. The assumption that $A$ belongs to $\mathcal{B}$ ensures $\sup (H\cap A)=\sup E\in A\subseteq A\cup B$. So $A\cup B\in \mathcal{B}$.
\end{proof}

\begin{theorem}
Let $X$ be an $\omega$-well-filtered space. Then $Y$ is an $\omega$-well-filtered subspace of $X$ if and only if $Y$ is strong$_*$ closed.
\begin{proof}
Assume $Y$ is an $\omega$-well-filtered subspace of $X$. We need to verify that $Y$ is strong$_*$ closed. For every $A\in \KF_{\omega}(Y)$ with $\sup A$ existing. In virtue of the $\omega$-well-filteredness of $Y$, we can identify some $x\in Y$ such that $\cl_Y(A)=\da_Y x$. This ensures that $\sup A=x\in Y$.

For the converse, let $Y$ be a strong$_*$ closed subset of $X$. We show that $Y$ is an $\omega$-well-filtered space. To this end, assume $A\in \KF_{\omega}(Y)$. We know $A\in \KF_{\omega}(X)$ since continuous mappings preserve $\KF_{\omega}$-sets. The fact that $X$ is $\omega$-well-filtered implies that $\sup A$ exists and $\cl_X(A)=\da_X \sup A$, see \cite{XiaoquanXu} for example. The strong$_*$ closedness of $Y$ implies that $\sup A\in Y$. This yields that $\cl_Y(A)=\da_{Y}\sup A$. Therefore, $Y$ is an $\omega$-well-filtered space.
\end{proof}
\end{theorem}

\begin{remark}
Because sober subspaces of $X$ $\omega$-well-filtered, it is clear that the strong topology is coarser than the strong$_*$ topology.
\end{remark}

By applying the above theorem and \cite[Theorem 4.4]{Keimel}, we arrive at the following theorem, which is the main result of this section.
\begin{theorem}
Let $X$ be a $T_0$ space. Then $\mathbf{WD}_{\omega}(X)=\cl_{s_*}(\eta_X^{\omega-w}(X))$, where $\cl_{s_*}(\eta_X^{\omega-w}(X))$ denotes the strong$_*$ closure of $\eta_X^{\omega-w}(X)$ in $X^s$.
\end{theorem}

\section{Sobriety and second countable quasi-metric spaces}
Every second countable $T_0$ space $X$ is a quasi-metrizable space.
Explicitly, such a (compatible) quasi-metric can be chosen to be Wilson's quasi-metric: $$ d(x,y)=\sup_{n\in \mathbb N,  x\in U_n, y\notin U_n} 1/ 2^n,$$
where $x, y\in X$ and $U_n, n\in \mathbb N$ are enumerations for a countable basis of $X$.  The reader could  refer \cite[Theorem 6.3.13]{GL} for more details. 
In \cite{XiaoquanXu2}, Xu et al. showed that $X^s=X^w=X^{\omega-w}$ for second countable $T_0$ spaces $X$, which is equivalent to saying that $X$ is $\omega$-well-filtered if and only if  it is well-filtered or sober when $X$ is second countable. In this section we look at those equivalence from a quasi-metric viewpoint. By means of Cauchy sequences with respect to the aforementioned (compatible) quasi-metric, we give characterizations for the standard $D$-completion, $\omega$-well-filterification, well-filterification, and sobrification when the underlying $T_0$ space is second-countable.
\begin{definition}
A countable subset $S$ of a $T_0$ space is said to be a \emph{$\Pi^{0}_2$-Cauchy subset} if $S\in \Pi^0_2(X)$ and  points of $S$ form a Cauchy sequence with respect to some compatible quasi-metic on $X$. The set of all $\Pi^{0}_2$-Cauchy subsets is denoted by $\Pi^{0}_2(CX)$.
\end{definition}
\begin{lemma}\label{Cauchy}
Let $X$ be a second countable $T_0$ space and $d$ a compatible quasi-metric on $X$ which generates the topology on $X$.
For every $A\in \mathbf{IRR}(X)$ with $A\notin \{\da_X x\mid x\in X\}$, there exists a $\Pi^{0}_2$-Cauchy subset $B$ of $A$ which is homeomorphic to $S_D$ or $S_1$ such that $\cl_X(B)=A$. 
\end{lemma}
\begin{proof}
Let $A\in \mathbf{IRR}(X)$ and $A\notin \{\da_X x\mid x\in X\}$. Then $\eta_{X}^{s}(A)=\{\da_X x\mid x\in A\}$ is irreducible in $X^s$ and $\cl_{X^s}(\eta_{X}^{s}(A)) =\da_{X^s} A$. Because $X$ is $A_2$, the space $X^s$ is $A_2$. It follows that there is a countable descending neighbourhood basis $(\mathcal{U}_n)_{n\in \mathbb{N}}$ of $A$. This means that $\mathcal{U}_{n}\cap \eta_{X}^{s}(A)\neq \emptyset$ for any $n\in \mathbb{N}$. Choose $\eta_{X}^{s}(a_n)\in \mathcal{U}_{n}\cap \eta_{X}^{s}(A)$ for each $n\in \mathbb{N}$. Write $\mathcal{B}=\{\eta_{X}^{s}(a_n)\mid n\in \mathbb{N}\}$. One sees immediately that $\cl_{X^s}(\mathcal{B})=\da_{X^s} A$.

\begin{enumerate}

\item $\cl_X(\{a_n\mid n\in \mathbb{N}\})=A$. \\
The implication from left to right is trivial, so assume $a\in A$. For any $a\in U$, we have $A\in U^{s} = \{B\in X^s \mid B\cap U \not= \emptyset\}$. It turns out that $U^{s} \cap \mathcal{B}\neq \emptyset$. This implies  that $U\cap \{a_n\mid n\in \mathbb{N}\}\neq \emptyset$. Set $H=\{a_n\mid n\in \mathbb{N}\}$. Note that the assumption $A\notin\{\da_X x\mid x\in X\}$ results in the infiniteness of $H$. It follows that $\cl_X(B)=\cl_X(H)=A$ for any infinite subset $B$ of $H$. 

\item 
Since $X$ is a quasi-metric space, we can know that $B_d(a_k, \frac{1}{n+1})_{n\in \mathbb{N}}$ is a countable descending neighbourhood basis of $a_k$ for any $k\in \mathbb{N}$. We inductively define an infinite sequence $\{x_{n} \} _{n\in \mathbb{N}}$ of distinct elements of $H$. Write $U_{0}=V_{0}=B_d(a_0,1)$ and let $x_{0}=a_0$.
Let $n\geq 0$, and assume that we have defined a sequence $x_{0},\cdots,x_{n}\in H$ and open sets $U_{0},\cdots, U_{n},V_{0}$, $\cdots, V_{n}$ with $V_{i}\subseteq U_{i},V_{i}\cap H\neq \emptyset$ for any $i\leq n$. We choose $U_{n+1},V_{n+1}$ as follows. We define
\begin{center}
	$U_{n+1}=V_{n}\backslash \bigcup_{i\leq n}\da x_{i}$;
	
	$V_{n+1}=U_{n+1}\cap\bigcap_{i\leq n}B_d(x_i, \frac{1}{n+2})$.
\end{center}
Then $V_{n+1}\cap H\neq \emptyset$ by the irreducibility of $H$. Pick $x_{n+1}$ from $V_{n+1}\cap H$. Thus, $S=\{x_{n}\mid n\in \mathbb{N}\}$ is an infinite subset of $H$.

\item The set $\{x_{n}\mid n\in \mathbb{N}\}$ froms a Cauchy sequence.\\
For any real number $ \varepsilon  > 0$, we let $n_0, m,n\in \mathbb{N}$ be such that $\frac{1}{n_0+1}<\varepsilon$ and $n\geq m\geq n_0$. Then $x_n\in B_d(x_m,\frac{1}{n+1})$. It means that $d(x_m,x_n)<\frac{1}{n+1}\leq \frac{1}{n_0+1}< \varepsilon$. Hence, $x_{n}, n\in \mathbb{N}$ is a Cauchy sequence.

\item $S\in \Pi^{0}_2(X)$.\\
Write $S_n=\{x_i\mid i\leq n\}\cup \bigcap_{i\leq n}V_i$ for $n\in \mathbb{N}$. It follows straightforward that $S_n\in \Pi^0_2(X)$ since $X$ is first-countable. So it suffices to check that $S=(\bigcap_{n\in \mathbb{N}}S_n)\cap \cl_X(S)$. The implication from left to right is trivial, so let $x\in (\bigcap_{n\in \mathbb{N}}S_n)\cap \cl_X(S)$. Assume for the sake of a contradiction that $x\notin S$. This means that $x\in \bigcap_{n\in \mathbb{N}}V_n$. Note that, for any $x_i\in S$, $\bigcap_{n\in \mathbb{N}}V_n\subseteq \bigcap_{n\in \mathbb{N}}B_d(x_i,\frac{1}{n+1})=\ua_X x_i$. It yields that $\cl_X(S)=\da _X x=\cl_X (H)=A$, which violates the fact that $A\notin\{\da_X x\mid x\in X\}$.
\end{enumerate}
For a similar proof of Theorem 4.3 in \cite{Brecht}, $S$ is a countable non-$T_{2}$ perfect $T_{D}$ space. Again by applying Theorem 3.5 in \cite{Brecht},  there is $B\in \Pi^{ 0}_{2}(S)$ homeomorphic to either $S_D$ or $S_1$. It turns out that $B\in \Pi^0_2(X)$ by~(4). As $B$ is infinite, $\cl_X(B)=\cl_X(H)=A$. The fact that each subsequence of a Cauchy sequence is a Cauchy sequence implies that $B$ is a Cauchy sequence. And the proof is complete. 
\end{proof}

Now the following theorems are easy to check. 

\begin{theorem}
Let $X$ be a second countable $T_0$ space and $d$ a compatible quasi-metric on $X$. Then $X^s=X^w=X^{\omega-w}=\{\cl_X(S)\mid S\in\Pi^{0}_2(CX)~\&~S=\{x\},x\in X\}\cup \{\cl_X(S)\mid S\in \Pi^{0}_2(CX)~\& ~S\cong S_D\}\cup \{\cl_X(S)\mid S\in \Pi^{0}_2(CX)~\&~ S\cong S_1\}$.
\end{theorem}

\begin{theorem}
Let $X$ be a second countable $T_0$ space and $d$ a compatible quasi-metric on $X$. Then $X^d=\{\cl_X(S)\mid S\in\Pi^{0}_2(CX)~\&~S=\{x\},x\in X\}\cup \{\cl_X(S)\mid S\in \Pi^{0}_2(CX)~\& ~S\cong S_D\}$.
\end{theorem}

\begin{theorem}
Let $X$ be a second countable $T_0$ space and $d$ a compatible quasi-metric on $X$. Then $X$ is sober if and only if for each Cauchy sequence $(a_n)_{n\in \mathbb{N}}$ in $X$ converges with respect to the metric topology $\tau_{\hat{d}}$ (see Section~\ref{qps-section}), where $S=\{a_n\mid n\in \mathbb{N}\}$ is homeomorphic to $S_D$ or $S_1$.
\end{theorem}
\begin{proof}
Assume that $X$ is sober. Then we need to show that for each Cauchy sequence $(a_n)_{n\in \mathbb{N}}$ in $X$ converges with respect to the metric topology $\tau_{\hat{d}}$, where $S=\{a_n\mid n\in \mathbb{N}\}$ is homeomorphic to $S_D$ or $S_1$. The fact that $S_D$ and $S_1$ are both irreducible indicates that $\cl_X(S)$ is an irreducible closed subset of $X$. Hence $\cl_X(S)=\da_X x_0$ for some $x_0\in X$, as $X$ is sober. Now we check that the Cauchy sequence $(a_n)_{n\in \mathbb{N}}$ converges to $x_0$ with respect to the metric topology $\tau_{\hat{d}}$.

For any real number $ \varepsilon  > 0$ with $x_0\in B_d(x_0,\varepsilon)$. Since $x_0\in \cl_X(S)$, the set $B_d(x_0,\varepsilon)\cap S$ is non-empty. As $S$ is homeomorphic to $S_D$ or $S_1$, we know that there is $N\in \mathbb{N}$ such that $a_n\in B_d(x_0,\varepsilon)$ for all $n\geq N$, that is, $d(x_0,a_n)<\varepsilon$. For each $n\geq N$, note that $a_n\leq x_0$ infers that $d(a_n,x_0)=0$. It follows that $a_n\in B_{\hat{d}}(x_0,\varepsilon)$ for all  $n\geq N$.

For the converse, assume for every Cauchy sequence $(a_n)_{n\in \mathbb{N}}$ in $X$ converges with respect to the metric topology $\tau_{\hat{d}}$, where $S=\{a_n\mid n\in \mathbb{N}\}$ is homeomorphic to $S_D$ or $S_1$. It remains to show that $X$ is sober. We proceed by contradiction. Suppose that there is $A\in \mathbf{IRR}(X)$ and $A\notin \{\da_X x\mid x\in X\}$. According to the construction of Lemma~\ref{Cauchy}, we can construct a Cauchy sequence $(a_n)_{n\in \mathbb{N}}$ with $S=\{a_n\mid n\in \mathbb{N}\}$ is homeomorphic to $S_D$ or $S_1$ and $\cl_X(S)=A$. In addition, the Cauchy sequence satisfies the property that $a_{n+1}\in \bigcap_{i\leq n}B_{d}(a_i,\frac{1}{n+2})$ for any $n\in \mathbb{N}$. By assumption, we can identify $x\in X$ such that the Cauchy sequence $(a_n)_{n\in \mathbb{N}}$ converges to $x$ with respect to the metric topology $\tau_{\hat{d}}$. 

We claim that $a_n\leq x$ for any $n\in \mathbb{N}$. 
For each fixed $n\in \mathbb{N}$, it suffices to show that $d(a_n,x)<\varepsilon$ for $ \varepsilon  > 0$. It turns out that there exists $N_1\in \mathbb{N}$ such that $N_1\geq n+1$ and $\frac{1}{N+1}<\frac{\varepsilon}{2}$. The fact that $(a_n)_{n\in \mathbb{N}}$ converges to $x$ with respect to the metric topology $\tau_{\hat{d}}$ guarantees the existence of $N_2$ such that $a_n\in B_{\hat{d}}(x,\frac{\varepsilon}{2})$ for all $n\geq N_2$. Let $N=\max\{N_1,N_2\}$. Then $d(a_n,x)\leq d(a_n,a_N)+d(a_N,x)<\frac{\varepsilon}{2}+\frac{\varepsilon}{2}=\varepsilon$.

The assumption that the Cauchy sequence $(a_n)_{n\in \mathbb{N}}$ converges to $x$ with respect to the metric topology $\tau_{\hat{d}}$ induces that $x\in \cl_X(S)$. This yields that $A=\cl_X(S)=\da _X x$, and this implies that $X$ is sober. 
\end{proof}

Inspecting the definition of quasi-Polish spaces, we can obtain the following corollary immediately, which has been shown by de Brecht in \cite{B}.

\begin{corollary}\cite{B}
Every quasi-Polish space is sober.
\end{corollary}


\bibliographystyle{plain}

\begin{thebibliography}{10}

\bibitem{B} M. de Brecht, Quasi-Polish spaces, Annals of Pure and Applied Logic. 164 (2013) 356-381.
\bibitem{Brecht} M. de Brecht, A generalization of a theorem of Hurewicz for quasi-Polish spaces, Logical Methods in Computer Science. 14 (2018) 1-18.

\bibitem{GHKLMS} G. Gierz, K. Hofmann, K. Keimel, J. Lawson, M. Mislove, D. Scott, Continuous Lattices and Domains, Cambridge University Press, 2003.

\bibitem{GL} J. Goubault-Larrecq, Non-Hausdorff Topology and Domain Theory, New Mathematical Monographs, Cambridge University Press, 2013.
\bibitem{Heckmann} R. Heckmann, An upper power domain construction in terms of strongly compact sets, Lecture Notes
in Computer Science. 598 (1991) 272-293.
\bibitem{Keimel} K. Keimel, J. Lawson, $D$-completions and the $d$-topology, Annals of Pure and Applied Logic. 159 (2007) 292-306.
\bibitem{LJMC} Q. Li, M. Jin, H. Miao, S. Chen, on some results related to sober spaces, Acta Mathematica Scientia. 43 (2023) 1477-1490.
\bibitem{Miao3} H. Miao, L. Wang, Q. Li, D-completion, well-filterification and sobrification, arXiv: 2101.04894.
\bibitem{Skula} L. Skula, On a reflective subcategory of the category of all topological spaces, Transactions American Mathematical Society. 142 (1969) 37–41.
\bibitem{Wu} G. Wu, X. Xi, X. Xu, D. Zhao, Existence of well-filterification, Topology and its Applications. 267 (2019) 107044.
\bibitem{XiaoquanXu} X. Xu, C. Shen, X. Xi, D. Zhao, First countability, $\omega$-well-filtered spaces and reflections, Topology and its Applications. 279 (2020) 107255.
\bibitem{XiaoquanXu2} X. Xu, C. Shen, X. Xi, D. Zhao, First-countability, $\omega$-Rudin spaces and well-filtered determined spaces, Topology and its Applications. 300 (2021) 107775.
\bibitem{XiaoquanXu3} X. Xu, C. Shen, X. Xi, D. Zhao, On $T_0$ spaces determined by well-filtered spaces, Topology and its Applications. 282 (2020) 107323.
\bibitem{zhongxizhang}Z. Zhang, Q. Li, A direct characterization of monotone convergence space completion, Topology and its Applications. 230 (2017) 99-104.








\end{thebibliography}

\end{document}